\documentclass{amsart}
\usepackage{amsfonts}
\usepackage{mathrsfs}
\usepackage[abs]{overpic}

\usepackage{palatino} 

\usepackage{hyperref}

\newtheorem{theorem}{Theorem}[section]
\newtheorem{lemma}[theorem]{Lemma}

\newtheorem{question}[theorem]{Question}
\theoremstyle{definition}

\def\R{\mathbb{R}}
\def\Z{\mathbb{Z}}

\def\dfn#1{{\em #1}}
\def\co{\colon\thinspace}
\newcommand{\dist}{\operatorname{dist}}

\theoremstyle{remark}
\newtheorem{remark}[theorem]{Remark}

\numberwithin{equation}{section}

\begin{document}

\title[The arc complex and contact geometry]{The arc complex and contact geometry:\\ non-destabilizable planar open book decompositions\\ of the tight contact 3-sphere}


\author{John Etnyre }
\address{School of Mathematics \\ Georgia Institute of Technology}
\email{etnyre@math.gatech.edu}
\urladdr{\href{http://www.math.gatech.edu/~etnyre}{http://www.math.gatech.edu/\~{}etnyre}}

\author{Youlin Li}
\address{Department of Mathematics \\ Shanghai Jiao Tong University}
\email{liyoulin@sjtu.edu.cn}

\begin{abstract}
In this note we introduce the (homologically essential) arc complex of a surface as a tool for studying properties of open book decompositions and contact structures. After characterizing destabilizability in terms of the essential translation distance of the monodromy of an open book we given an application of this result to show that there are planer open books of the standard contact structure on $S^3$ with 5 (or any number larger than 5) boundary components that do not destabilize. We also show that any planar open book with 4 or fewer boundary components does destabilize. 
\end{abstract}

\maketitle
\section{Introduction}

The Giroux correspondence between open book decompositions and contact structures on 3--manifolds has been a key tool in contact geometry and low dimensional topology for quite some time. In studying relations between open book decompositions and contact structures one frequently wants to know if the open book decomposition is in some sense ``minimal'' or at least ``non-destabilizable". The main theme explored in this note is destabilizability of open book decompositions. 

The bindings of open book decompositions of $S^3$ are in fact fibered links. The construction and study of fibered links in $S^3$ has a long history, see for example \cite{GirouxGoodman06, Harer82, MelvinMorton86, Stallings78}.   Hedden  \cite{Hedden10} proved (Baader and Ishikawa \cite{BaaderIshikawa09} later independently re-proved and Baker, Etnyre and Van Horn-Morris generalized \cite{BakerEtnyreVanHorn-Morris12})  that a fibered link in $S^3$ supports the unique tight contact structure on $S^3$ if and only if it is quasipositive. Also in \cite{BaaderIshikawa09}, Baader and Ishikawa raised the following question: Does there exist a quasipositive fiber surface, other than a disk, from which we cannot de-plumb a Hopf band? In the terminology of this paper, that is: Does there exist an open book decomposition that supports the standard tight contact structure on $S^3$, other than the $(D^2, \text{Id})$, that does not destabilize?

In \cite{BakerEtnyreVanHorn-Morris12}, Baker, Etnyre, and Van Horn-Morris and in \cite{Wand??}, Wand found  examples of open book decompositions that support the standard tight contact structure on $S^3$, have page genus $2$, and cannot be destabilized down to the trivial open book $(D^2, \text{Id})$.  Thus there are non-destabilizable examples of genus 2, or 1.  (We note the results in this paper imply that if their examples destabilized to a planar open book then they would destabilize to the trivial open book, thus their examples prove the existence of a genus 2 or 1 non-destabilizable open book.)

In this paper, we completely answer this question in the planar case. We first note the existence of non-destabilizable examples with 5 or more boundary components.  

\begin{theorem}\label{thm1}
For each integer $n\geq 5$, there is a planar open book decomposition $(\Sigma_{n}, \phi_{n})$ which supports the tight contact structure on $S^3$,  has $n$ binding components,  and cannot be destabilized.
\end{theorem}

However, if the binding number is small enough, then the planar open books can be destabilized.

\begin{theorem}\label{thm2} Suppose $(\Sigma, \phi)$ is a planar open book decomposition which supports the standard tight contact structure on $S^3$ other than the $(D^2, \text{Id})$, and has at most $4$ binding components, then $(\Sigma, \phi)$ can be destabilized.  
\end{theorem} 

We leave the following question open.
\begin{question}
Are there open book decompositions supporting the standard tight contact structure on $S^3$ with page genus $1$ that cannot be desabilzied? What about higher genus?
\end{question}

One of the main tools we use in proving the above theorems is the homologically essential arc complex and essential translation distance. The arc complex was introduced by Harer in \cite{Harer82} to study the mapping class group but was first used to study open books in Saito and Yamamoto's \cite{SaitoYamamoto10}. Here we introduce the homologically essential arc complex and show its efficacy in studying stabilization of open books. We hope these new tools will find more applications in contact geometry. One can define the arc complex of a surface with boundary whose vertices are, more or less, isotopy classes of properly embedded arcs and two vertices have an edge if they are represented by disjoint arcs, see Subsection~\ref{ssec:arcc} for precise definitions. There is also the homologically essential arc complex defined in the same way, but only using homologically essential arcs. A diffeomorphism $\phi$ of the surface acts on both these complexes and we define the translation distance $\dist(\phi)$ and the essential translation distance $\dist_e(\phi)$ to be the minimal distance $\phi$ moves vertices in the respective complexes. We make the following observations.
\begin{lemma}\label{dist1}
Let $(\Sigma,\phi)$ be an open book decomposition. Then $\dist_e(\phi)=0$ implies that $M_{(\Sigma,\phi)}$ has an $S^2\times S^1$ summand. 
\end{lemma}

\noindent
We note the examples from \cite{Wand??} imply that the other implication is not true.

\begin{theorem}\label{dist2}
An open book decomposition for a tight contact structure on any manifold without an $S^2\times S^1$ summand destabilizes if and only if its essential translation distance is 1.
\end{theorem}

Generalizing previous conjectures, now known to be false, we can ask the following question. 
\begin{question}
Is there some integer $n$ larger than 0 such that 
if an open book decomposition is right veering and its monodromy has essential translation distance greater than $n$, then the supported contact structure is tight? The Heegaard-Floer contact invariant is non-zero?
\end{question}
We end with one last question.
\begin{question}
Is there a relation between fractional Dehn twist coefficients of a monodromy map and its (essential) translation distance?
\end{question}

{\bf Acknowledgements}:  We thank Meredith Casey and Dan Margalit for useful conversations during the preparation of this paper. The first author thanks the Simons Center for their hospitality during the completion of the work in this paper. The first author was partially supported by NSF grant DMS-0804820 and DMS-1309073.  The second author was partially supported by NSFC grant 11001171 and the China Scholarship Council grant 201208310626.  This work was carried out while the second author was visiting Georgia Institute of Technology and he would like to thank for their hospitality.

\section{Background and preliminary notions}

In this section we begin by recalling the definition of an open book decomposition, its associated contact structure and other relevant notions. In the following subsection we define the arc complex and the homologically essential arc complex, define a notion of distance in these complexes and in the last subsection give a simple way one can try to bound the distance.

\subsection{Open book decompositions and contact structures}\label{ssec:obd}
Recall that given a surface $\Sigma$ with boundary and a diffeomorphism  $\phi\co \Sigma\to\Sigma$ that fixes the boundary, we can construct a 3--manifold $M_{(\Sigma,\phi)}$ by collapsing each circle $x\times [0,1]/\sim$ in the boundary of the mapping torus
\[
T_\phi:\Sigma\times[0,1]/(x,1)\sim(\phi(x),0)
\]
to a point. We call the image of $\Sigma\times \{t\}$ a \dfn{page} of the open book and the boundary of this surface is called the \dfn{binding} of the open book. The diffeomorphism $\phi$ is called the \dfn{monodromy} of the open book. If $M$ is a 3--manifold diffeomorphic to $M_{(\Sigma,\phi)}$ then we say that $(\Sigma, \phi)$ is an \dfn{open book decomposition} for $M$. See \cite{Etnyre06} for more details. 

A \dfn{positive stabilization} of an open book decomposition $(\Sigma, \phi)$ of a manifold $M$ is the open book decomposition $(\Sigma', \phi')$ obtained as follows: let $\Sigma'$ be $\Sigma$ with a 1--handle attached and let $\alpha$ be a curve in $\Sigma'$ that (transversely) intersects the co-core of the new 1--handle exactly once. Set $\phi'=\phi\circ \tau_\alpha$, where $\tau_\alpha$ is the right handed Dehn twist along $\alpha$. The \dfn{negative stabilization} of $(\Sigma,\phi)$ is defined in the same way except one set $\phi'=\phi\circ \tau^{-1}_\alpha$. One may easily check that $M_{(\Sigma',\phi')}$ is diffeomorphic to $M_{(\Sigma,\phi)}$ for both the positive and negative stabilization. 

It is well known \cite{Giroux02, ThurstonWinkelnkemper75} that to an open book decomposition of $M$ there is a naturally associated contact structure $\xi_{(\Sigma,\phi)}$ and we say that the open book $(\Sigma, \phi)$ \dfn{supports} this contact structure. Moreover positive stabilization does not change the contact structure where as negative stabilization does. 

In \cite{HondaKazezMatic07} the notion of an open book decomposition being right veering was defined as a way of trying to better understand whether or not the associated contact structure was tight or overtwisted.  We recall the definition here. Let $\gamma$ and $\gamma'$ be two arcs properly embedded in an oriented surface $\Sigma$ such that they have a common endpoint $x$. Isotope the curves (rel boundary) so that they intersect minimally. We say that $\gamma$ is \dfn{to the right} of $\gamma'$ at $x$ if the tangent vector to $\gamma$ at $x$ followed by the tangent vector of $\gamma'$ at $x$ forms an oriented basis for $T_x\Sigma$, we also say that $\gamma'$ is to the left of $\gamma$ at $x$. A diffeomorphism $\phi$ of $\Sigma$ that fixes the boundary is called \dfn{right veering} if for every arc $\gamma$ that is properly embedded in $\Sigma$, the image of $\gamma$ is either isotopic to $\gamma$ or is to the right of $\gamma$ at each endpoint. The fundamental observation of Honda, Kazez and Mati\'c in \cite{HondaKazezMatic07} was that if $\xi$ is a tight contact structure then all of the open books that support it will be right veering. 

\subsection{The arc complex}\label{ssec:arcc}
Given a surface $\Sigma$ with non-empty boundary. Choose a distinguished point on each boundary component. We will consider properly embedded arcs with endpoints on a subset of these distinguished points. We call a properly embedded arc \dfn{essential} if it is not isotopic into the boundary of $\Sigma$. 

\smallskip
\noindent
{\em Important convention:} When discussing embedded arcs we only require that they are imbedded on their interior. That is both the endpoints can be mapped to the same place. 
\smallskip

We define the \dfn{arc complex} of $\Sigma$, denoted $\mathcal{A}(\Sigma)$, to be the complex with vertices the isotopy classes of properly embedded essential arcs with endpoints at the chosen distinguished points on the boundary. There will be an edge between two {\em distinct} vertices if the interiors of the arcs are disjoint after they have been isotoped to intersect minimally. Finally $k$ distinct vertices will bound a $k$-simplex if they can be isotoped to have disjoint interior. 

We note that this definition is essentially equivalent to Harer's definition in \cite{Harer86}, though he allows for marked points on the interior, and is similar to Saito and Yamamoto's definition in \cite{SaitoYamamoto10}. We did not check if our definition is equivalent to Saito and Yamamoto's, but suspect that this is the case. 
We note that even though Saito and Yamamoto were also studying open book decompositions we think our definition that requires arcs to have fixed endpoints  is more convenient with discussing various notions such as right veering that have become important in contact geometry. 

There is a natural subcomplex $\mathcal{A}_e(\Sigma)$ of $\mathcal{A}(\Sigma)$ who vertices consist of homologically essential arcs and the higher cells are defined as above. We call this the \dfn{homologically essential arc complex}. 

We define a distance function on the 0-skeleton of $\mathcal{A}(\Sigma)$ 
\[
d\co\mathcal{A}^0(\Sigma)\times \mathcal{A}^0(\Sigma)\to \R
\]
by declaring each edge in $\mathcal{A}$ to have unit length. In particular $d([\gamma],[\gamma'])$ is $k$ if $k$ is the smallest integer for which there is a sequence of distinct vertices $\gamma_0,\ldots, \gamma_k$ such that $\gamma=\gamma_0$, $\gamma'=\gamma_k$ and $\gamma_i$ is connected to $\gamma_{i+1}$ by an edge for each $i=0,\ldots, k-1$. It is useful to point out that while Saito and Yamamoto's definition of the arc complex is different from ours, it is clear that the vertices are in one-to-one correspondence and under this correspondence the distance function we just defined is the same as theirs. Thus the notion of translation distance below is also the same as theirs. 

Similarly we define a distance function on the 0-skeleton of $\mathcal{A}_e(\Sigma)$ 
\[
d_e\co\mathcal{A}_e^0(\Sigma)\times \mathcal{A}_e^0(\Sigma)\to \R
\]
as above. 
\subsection{A simple bound on distance}

Computing  distance can be very difficult, so in this subsection we indicate how to get bounds on certain distances. 
Given a surface $\Sigma$ with more than one boundary component we say $\Sigma'$ is obtained from it by \dfn{capping off a boundary component} if $\Sigma'$ is obtained from $\Sigma$ by gluing a disk to one of its boundary components. 
\begin{lemma}\label{lem:b1}
Let $\Sigma$ be a surface with more than one boundary component. If $\Sigma'$ is obtained from $\Sigma$ by capping off a boundary component and the distance between $\gamma$ and $\gamma'$ in $\mathcal{A}(\Sigma')$ is greater than 1 then the distance in $\mathcal{A}(\Sigma)$ is also greater than 1. The same statement holds in the homologically essential arc complex.  
\end{lemma}

\begin{proof}
Given arcs $\gamma$ and $\gamma'$ in $\Sigma$ (without endpoints on the capped off boundary  component) that have distance 1 or 0 as arcs in $\Sigma$ we see that distance 0 implies they are isotopic and hence they would be isotopic in $\Sigma'$ too (and hence distance 0 there). Distance 1 implies they are not isotopic but have disjoint interiors which clearly implies they have distance 0 or 1 in $\Sigma'$. Thus if the distance in $\Sigma'$ is greater than 1, then it must also be great than 1 in $\Sigma$. 
\end{proof}


\section{Translation distance and destabilization}
In this section we will define the translation distance and essential translation distance of a surface diffeomorphism and show how the later relates to destabilization of open book decompositions. 

\subsection{Translation distance}

Given a diffeomorphism $\phi$ of a surface with boundary $\Sigma$ (that fixes the boundary) we define the \dfn{translation distance} to be the minimal distance that $\phi$ moves a vertex in $\mathcal{A}(\Sigma)$:
\[
\dist(\phi)=\min\{d(\alpha, \phi(\alpha)):\alpha \in \mathcal{A}^0(\Sigma)\}.
\]
This notion was originally defined in \cite{SaitoYamamoto10}  using a slightly different notion for the arc complex, but as commented in Subsection~\ref{ssec:arcc} the notion of distance, and hence translation distance, is the same. We similarly define the \dfn{essential translation distance}:
\[
\dist_e(\phi)=\min\{d_e(\alpha, \phi(\alpha)):\alpha \in \mathcal{A}_e^0(\Sigma)\}.
\]

We will see below that the essential translation distance has a closer connection to contact geometric properties and for that reason we restrict attention to it. We gave the definition of translation distance largely to tie this work with that of Saito and Yamamoto and to contrast it with essential translation distance. For now we observe in the next two lemmas one simple difference between the two different notions of translation distance, the first lemma is just Lemma~\ref{dist1} from the introduction. 
\begin{lemma}\label{lem:de0}
Let $(\Sigma,\phi)$ be an open book decomposition. Then $\dist_e(\phi)=0$ implies that $M_{(\Sigma,\phi)}$ has an $S^2\times S^1$ summand.
\end{lemma}
\begin{proof}
Suppose that $\dist_e(\phi)=0$. In this case there is an essential arc $\gamma$ in $\Sigma$ such that $d_e(\gamma,\phi(\gamma))=0.$ That is $\phi$ can be assumed to fix $\gamma$. Since $\gamma$ is essential there is an embedded closed curve $\gamma'$ in $\Sigma$ that intersects $\gamma$ exactly once and the intersection is transverse. Notice that $\gamma\times [0,1]$ in $M_{(\Sigma,\phi)}$ is a 2--sphere $S$ (see Subsection~\ref{ssec:obd} to recall the definition of $M_{(\Sigma,\phi)}$). We can think of $\gamma'$ as sitting on one of the pages of the open book and this gives a simple closed curve in $M_{(\Sigma,\phi)}$ that intersects $S$ exactly once. Let $N$ be a neighborhood of $S\cup \gamma'$ in $M_{(\Sigma,\phi)}$ and $N'$ be the closure of its complement. One can glue a closed 3--ball to $N$ and $N'$ to obtain closed 3--manifolds $M$ and $M'$. It is now easy to see that $M_{(\Sigma,\phi)}=M\# M'$ and that $M=S^2\times S^1$. 
\end{proof}
\begin{lemma}
Let $(\Sigma,\phi)$ be an open book decomposition. Then $\dist(\phi)=0$ implies that either 
\begin{enumerate}
\item $M_{(\Sigma,\phi)}$  has an $S^2\times S^1$,
\item $M_{(\Sigma,\phi)}$ can be written as a non-trivial connected sum of two manifolds, or
\item $\Sigma$ can be written as the boundary sum of two surface $\Sigma_1\natural \Sigma_2$, $\phi$ can be isotoped to preserve $\Sigma_1$ and $\Sigma_2$ and $(\Sigma_i,\phi_i)$ is an open book for $S^3$ for $i=1$ or $2$, where $\phi_i=\phi|_{\Sigma_i}$. 
\end{enumerate}
\end{lemma}
The proof of this lemma is very similar to the proof of the previous lemma and is left to the reader. 
So we see from these two lemmas that the essential translation distance can tell us about the manifold $M_{(\Sigma,\phi)}$ whereas the translation distance can only tell us about the open book decomposition. 

\subsection{Destabilizing open book decompositions}

The connection between the essential translation distance of diffeomorphism and destabilization of an open book decomposition is given in the following theorem. 
\begin{theorem}
Let $(\Sigma,\phi)$ be an open book decomposition that is right veering. If the 3--manifold $M_{(\Sigma,\phi)}$ associated to the open book does not have an $S^2\times S^1$ summand, then $(\Sigma,\phi)$ (positively) destabilizes if and only if $\dist_e(\phi)= 1$. 
\end{theorem}

We note that the obvious analogous theorem for left veering monodromies and negative destabilization also holds. The proof is left to the reader. Before proving this theorem we note that Theorem~\ref{dist2} from the introduction is an obvious corollary.

\begin{proof}
We first assume that $(\Sigma,\phi)$ destabilizes. Thus there is an open book $(\Sigma',\phi')$ of which $(\Sigma,\phi)$ is a stabilization. More specifically $\Sigma$ is $\Sigma'$ with a 1--handle attached and $\phi=\phi'\circ \tau_\alpha$ for some $\alpha$ that intersects the co-core of the attached 1--handle once. Let $\gamma$ denote this co-core. 

Now in the definition of the arc complex choose the marked points on the boundary to contain $\partial \gamma$ if $\partial \gamma$ is contained in two separate boundary components and to be one of the endpoints of $\gamma$ otherwise. In the latter case isotope one of the endpoints of $\gamma$ so that it is also at a marked point. One may easily check that $d_e(\gamma,\phi(\gamma))=1$. Thus $\dist_e(\phi)\leq 1.$ If $\dist_e(\phi)=0$ then $M_{(\Sigma,\phi)}$ would have an $S^2\times S^1$ summand by Lemma~\ref{lem:de0} and this is ruled out by the hypothesis of the theorem. Thus $\dist_e(\phi)= 1.$

Now suppose that $\dist_e(\phi)=1$. Then there is an arc $\gamma$ such that $d_e(\gamma,\phi(\gamma))=1$. That is $\gamma$ and $\phi(\gamma)$ have disjoint interiors. If $\gamma$ has both its endpoints at the same point then we can clearly isotope $\gamma$ (moving one of its endpoints) so that $\gamma$ is properly embedded and has distinct endpoints (recall our {\em important convention} above) and $\gamma$ and $\phi(\gamma)$ are still disjoint. Now let $\alpha$ be the simple closed curve in $\Sigma$ formed from $\gamma\cup \phi(\gamma)$ by connecting their endpoints and isotoping so that the curve is in the interior of $\Sigma$. Notice that due to the right veeringness of $\phi$ it is clear that we can choose $\alpha$ so that it intersects both $\gamma$ and $\phi(\gamma)$ each exactly once. One may easily verify that $\tau^{-1}_\alpha\circ \phi$ fixes $\gamma$. Thus we may form a surface $\Sigma'$ by cutting along $\gamma$ and $\tau_\alpha^{-1}\circ\phi$ will induce a diffeomorphism $\phi'$ of $\Sigma'$. It is also clear that $(\Sigma,\phi)$ is a stabilization of $(\Sigma',\phi')$.
\end{proof}
\begin{remark}
We note that the proof actually shows that, under the hypothesis of the theorem, an arc is the co-core of a 1--handled used to stabilize an open book if and only if the essential distance between it and its image under $\phi$ is 1. The co-core of a stabilizing 1--handle will be called a \dfn{stabilizing arc}. 
\end{remark}

\section{Non-destabilizable planar open book decompositions of $S^3$}

In this section we will prove Theorem~\ref{thm1}. To that end let $\Sigma_n$ be a compact planar surface with $(n+1)$-boundary components, $n\geq 4$, and $\phi_n\co \Sigma_n\to \Sigma_n$ be a diffeomorphism obtained as the composition of right handed Dehn twists along the curves shown in Figure~\ref{fig:sigman}. Since all other curves, except one, are disjoint, the conjugacy class of $\phi_n$ is independent of the order of the product.  Let $c_i$ denote the boundary components of $\Sigma_n$ as indicated in Figure~\ref{fig:sigman}.

\begin{figure}[htb]
\begin{overpic}
{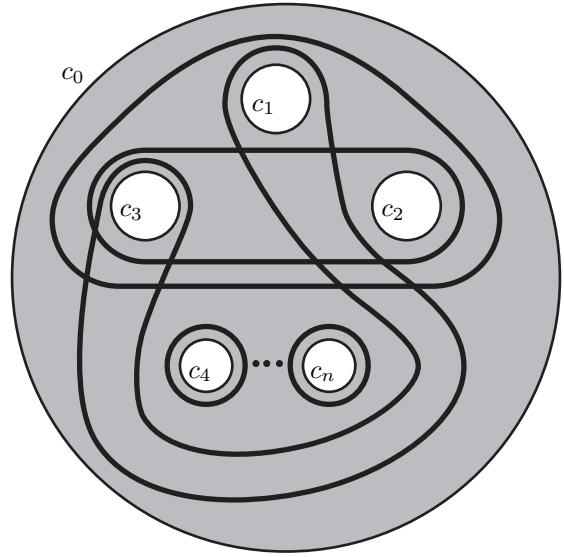}
\put(20, 180){$c_0$}
\put(92,168){$c_1$}
\put(42, 128){$c_3$}
\put(141, 128){$c_2$}
\put(68, 67){$c_4$}
\put(114,67){$c_n$}
\end{overpic}
\caption{The surface $\Sigma_n$ with boundary components $c_0,\ldots, c_n$ labeled.}
\label{fig:sigman}
\end{figure}

\begin{lemma}
The open book $(\Sigma_n, \phi_n)$ supports the standard tight contact structure on $S^3$. 
\end{lemma}
\begin{proof}
Since $\phi_n$ is a composition of right handed Dehn twists it is well known that the supported contact structure is tight, \cite{Giroux02}. Since $S^3$ has a unique tight contact structure, \cite{Eliashberg92a}, we are left to see that $(\Sigma_n, \phi_n)$ is an open book for $S^3$. One may readily verify that Figure~\ref{fig:surgs3} is a Kirby diagram for $M_{(\Sigma_n,\phi_n)}$. After sliding the left most 0--framed unknot over right most 0--framed unknot (first isotope the right unknot to be concentric with the left unknot) 
\begin{figure}[htb]
\begin{overpic}
{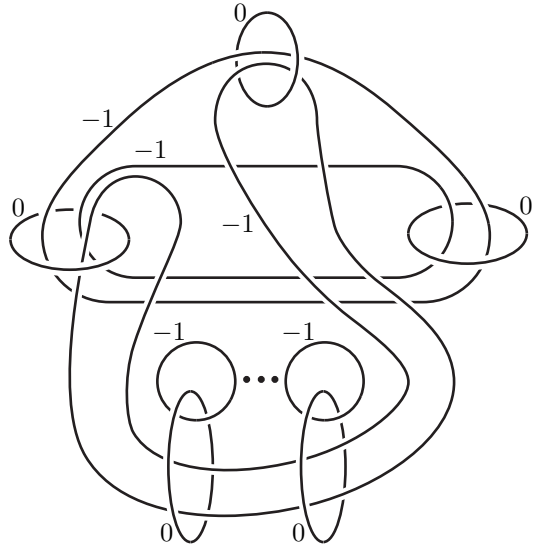}
\put(4,126){$0$}
\put(196,127){$0$}
\put(88,200){$0$}
\put(60,3){$0$}
\put(110,3){$0$}
\put(57, 79){$-1$}
\put(106, 79){$-1$}
\put(30, 160){$-1$}
\put(50, 148){$-1$}
\put(83, 120){$-1$}
\end{overpic}
\caption{Surgery diagram for the manifold described by the open book decomposition $(\Sigma_n,\phi_n)$.}
\label{fig:surgs3}
\end{figure}
one sees that all but one non-zero framed curve has a zero framed meridian and so may be cancelled form the picture. Once this is done we are left with a Hopf link whose components are framed 0 and $-1$, thus giving $S^3$. 
\end{proof}

To see that $(\Sigma_n,\phi_n)$ does not destabilize we compute the essential translation distance. 
\begin{theorem}\label{etd2}
With the notation above 
\[
\dist_e(\phi_n)=2.
\]
\end{theorem}
Before proving this we establish Theorem~\ref{thm1}. 
\begin{proof}[Proof of Theorem~\ref{thm1}]
The proof immediately follows from Theorem~\ref{etd2} and Theorem~\ref{dist2}. 
\end{proof}

\begin{proof}[Proof of Theorem~\ref{etd2}]
First note that if $\gamma$ is a horizontal arc connecting boundary $c_0$ to $c_2$ then \[d_e(\gamma,\phi_n(\gamma))=2\] since they intersect but it is easy to find an arc disjoint from both. Thus $\dist_e(\phi_n)\leq 2$. We are left to see it is not $1$ or 0. Of course it cannot be zero by Lemma~\ref{lem:de0}. We will see that it is not $1$ in Lemmas~\ref{lem4}, \ref{lem2} and~\ref{lem3}  below. More specifically the lemmas show that for any arc $\gamma$ we have $d_e(\gamma,\phi_n(\gamma))>1$. Notice that we only need to check arcs connecting different boundary components since on a planar surface a stabilizing arc cannot have both endpoints on the same boundary component since it would separate the surface.
\end{proof}


\begin{lemma}\label{lem4}
Let $\gamma$ be any arc connecting $c_0$ to $c_2$ in $\Sigma_n$, then $d_e(\gamma,\phi_n(\gamma))>1$. Similarly any arc $\gamma$ connecting $c_i$ to $c_j$ with $4\leq i\not=j \leq n$ will have $d_e(\gamma,\phi_n(\gamma))>1$.
\end{lemma}
\begin{proof}
Let $(\Sigma',\phi')$ be the open book obtained form $(\Sigma_n,\phi_n)$ by capping off all boundary components except $c_0$ and $c_2$. Notice that $\Sigma'$ is an annulus and $\phi'$ is the square of the Dehn twist about the core curve in $\Sigma'$. Thus it is clear that $\dist_e(\phi')>1$ and hence by Lemma~\ref{lem:b1} we know $d_e(\gamma,\phi_n(\gamma))>1$. 
The other cases are follow similarly. 
\end{proof}

\begin{lemma}\label{lem2}
Let $(\Sigma',\phi')$ be obtained from $(\Sigma_n,\phi_n)$ by capping off $c_2$. Then $\dist_e(\phi')>1$. In particular, any arc $\gamma$ in $\Sigma_n$ with endpoints on any boundary component except $c_2$ must satisfy $d_e(\gamma,\phi_n(\gamma))>1$.
\end{lemma}

\begin{proof} We will show that the open book obtained from $(\Sigma', \phi')$ by capping off all but one of the boundary components $c_4,\ldots, c_n$ satisfies $\dist_e(\phi'')>1$. Combining this with Lemma~\ref{lem:b1} and~\ref{lem4} we will clearly have that $\dist_e(\phi')>1$. 

The open book decomposition $(\Sigma'', \phi'')$ is shown on the lefthand side of Figure~\ref{fig:obsurg2}. A Kirby diagram for the corresponding manifold is shown on the righthand side of Figure~\ref{fig:obsurg2} 
\begin{figure}[htb]
\begin{overpic}
{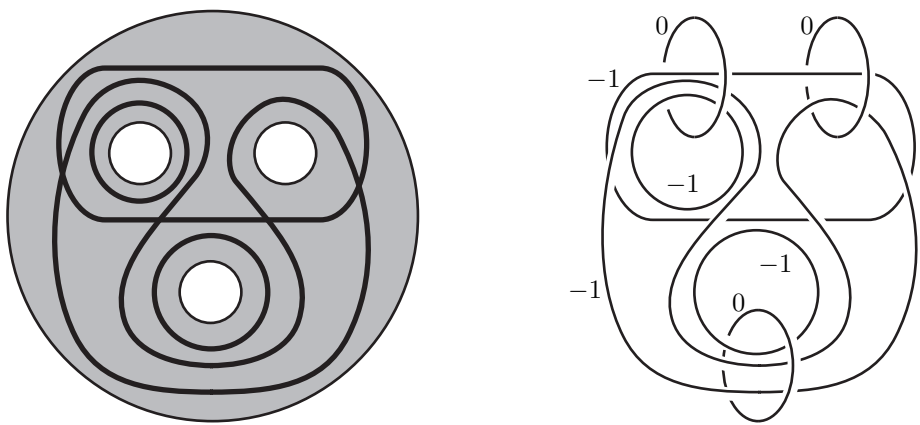}
\put(246,150){$0$}
\put(301,150){$0$}
\put(275,45){$0$}
\put(220,130){$-1$}
\put(213,50){$-1$}
\put(285,60){$-1$}
\put(250,90){$-1$}
\end{overpic}
\caption{The open book decomposition $(\Sigma'',\phi'')$ on the lefthand side. On the righthand side is a Kirby picture for the manifold supported by $(\Sigma'',\phi'')$.}
\label{fig:obsurg2}
\end{figure}
and an easy exercise in Kirby calculus yields that this manifold is the one obtained from $(-2)$--surgery on the right handed trefoil in $S^3$. Since all the Dehn twist are positive we also see that the contact structure is Stein fillable and hence tight. We claim that the open book is not a stabilization. If it were then it would be a stabilization of an open book decomposition $(P, \psi)$,  where $P$ is a pair of pants and $\psi$ is a product of some right handed Dehn twists, since if one ever used a left handed Dehn twist the monodromy would be left veering and the contact structure would be overtwisted \cite{HondaKazezMatic07}.  Label the boundary components of $P$ and suppose $\psi=\tau_{0}^{p}\tau_{1}^{q}\tau_{2}^{r}$, where $\tau_{i}$ is the right handed Dehn twist along a simple closed curve parallel to the $i$th boundary component,  and $p,q,r$ are all nonnegative integers. Then the open book decomposition $(P, \psi)$ supports the 3-manifold obtained by 0--surgery on the unknot followed by $p, q$ and $r$ surgeries on three meridians to the unknot. The first homology of this manifold has size $pq+qr+rp$. On the other hand, the first homology of $(-2)$--surgery on the trefoil is $\mathbb{Z}_{2}$. So $pq+qr+rp=2$, in particular $\{p,q,r\}=\{0,1,2\}$. Hence $(P, \psi)$ supports the lens space $L(2,1)$. It is well known, see \cite{Moser71}, that the only surgeries on the right handed trefoil knot yielding a Lens space are the 5 and 7 surgeries (in addition one can easily check that the fundamental group the manifold described by $(\Sigma'',\phi'')$ is not $\Z_2$). Thus  $(\Sigma'', \phi'')$ does not destabilize. 
\end{proof} 

\begin{lemma}\label{lem3}
Let  $(\Sigma',\phi')$ be obtained from $(\Sigma_n,\phi_n)$ by capping off $c_0$. Then $\dist_e(\phi')>1$. In particular, any arc $\gamma$ in $\Sigma_n$ with endpoints on any boundary component except $c_0$ must satisfy $d_e(\gamma,\phi_n(\gamma))>1$.
\end{lemma}

\begin{proof} 
Arguing as in the proof of Lemma~\ref{lem2} it suffices to consider the open book $(\Sigma'', \phi'')$ obtained from $(\Sigma', \phi')$ by capping off all but one of the boundary components $c_4,\ldots, c_n$ satisfies $\dist_e(\phi'')>1$. That open book decomposition is shown on the lefthand side of Figure~\ref{fig:ob3} and supports 
\begin{figure}[htb]
\begin{overpic}
{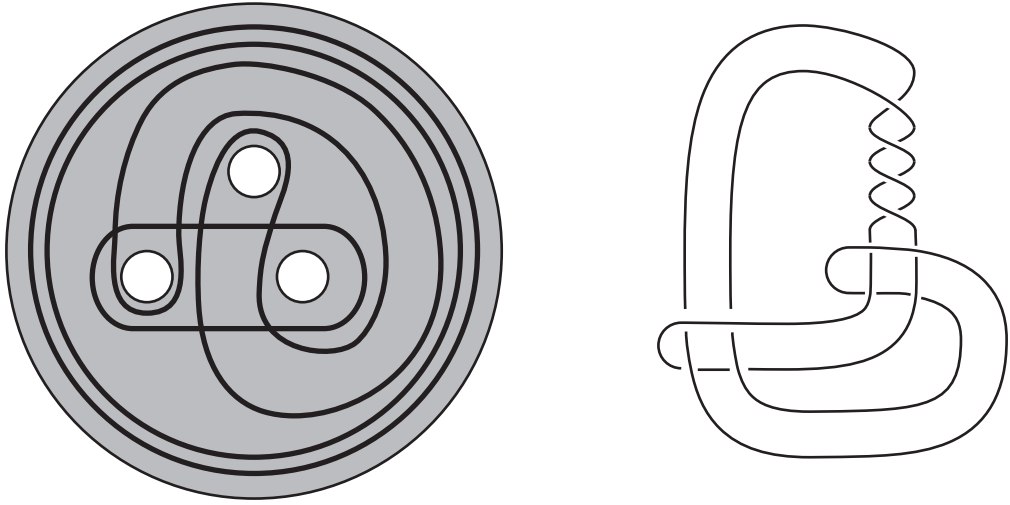}
\end{overpic}
\caption{The open book decomposition $(\Sigma'',\phi'')$ on the lefthand side. On the righthand side is the knot $K$.}
\label{fig:ob3}
\end{figure}
a Stein fillable contact structure on the 3--manifold. One can produce a Kirby picture for this manifold as we have done above and then an exercise in Kirby calculus, see Figure~\ref{kc},  shows this manifold is obtained from 2--surgery on the knot $K$ shown on the righthand side of Figure~\ref{fig:ob3}. 
\begin{figure}[htb]
\begin{overpic}
{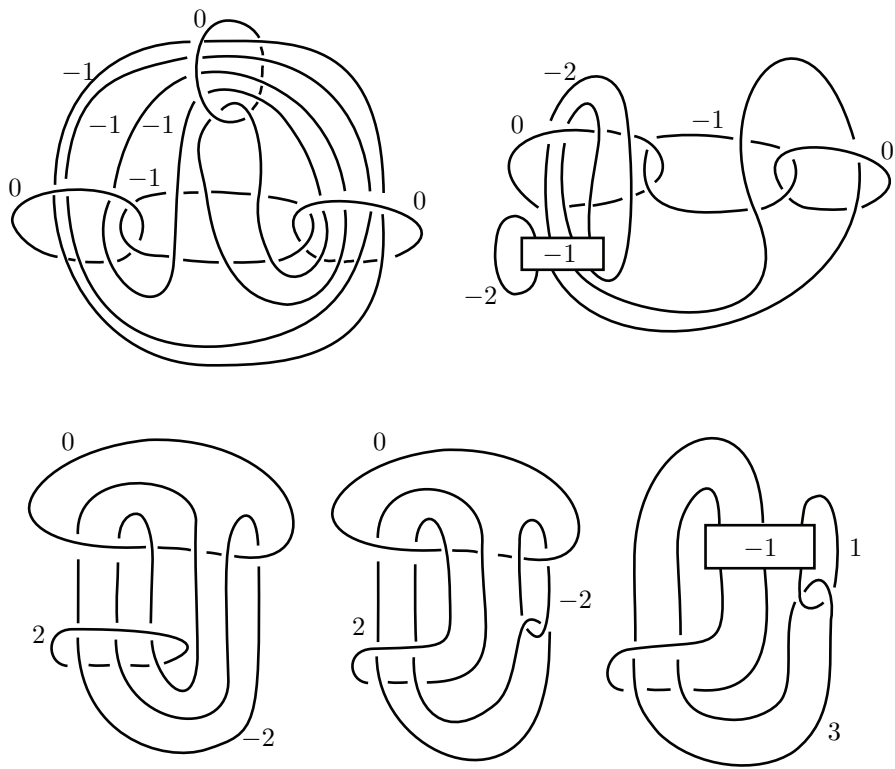}
\put(0,216){$0$}
\put(70,280){$0$}
\put(153, 211){$0$}
\put(20,260){$-1$}
\put(30, 240){$-1$}
\put(50,240){$-1$}
\put(45,220){$-1$}
\put(202,191){$-1$}
\put(172,175){$-2$}
\put(202, 260){$-2$}
\put(190,240){$0$}
\put(330, 230){$0$}
\put(258, 242){$-1$}
\put(9,47){$2$}
\put(88,8){$-2$}
\put(20,120){$0$}
\put(130, 50){$2$}
\put(208, 60){$-2$}
\put(138,120){$0$}
\put(310, 10){$3$}
\put(318, 80){$1$}
\put(278, 80){$-1$}
\end{overpic}
\caption{The Kirby Calculus to get from a surgery presentation of the manifold supported by $(\Sigma'', \phi'')$, shown in the upper left, to $2$ surgery on the knot $K$. To go from the upper left figure to the upper right figure, slide the outer most $(-1)$--framed 2--handle over the parallel one, also slide the non-oval $(-1)$--framed 2--handle over the this 2--handle three times (twice over arcs contained in the leftmost 0--framed 2--handle and once over an arc in the uppermost  0--framed 2--handle), and finally cancel two 2--handles. To go from the upper right figure to the lower left, slide the non-oval 2--handle over the $(-1)$--framed 2--handle and cancel two 2--handles. To get to the middle diagram on the bottom row, just isotope. The final diagram on the lower right is then obtained by blowing up a $(+1)$--framed unknot to unlink the 0 and $(-2)$--framed handles, blowing down the resulting $(-1)$--framed handle and one of the $(+1)$--framed handles. Finally blowing down the $(+1)$--framed unknot in the lower right diagram yields 2--surgery on the knot $K$.}
\label{kc}
\end{figure}

The fundamental group of this manifold can be presented as
\[
\langle   a,b,c|
   acb^{-1}a^{-1}c^{-1}acbc^{-1},
   aab^{-1}cbabc^{-1},
   aba^{-1}b^{-1}cc\rangle.
\]
(This may be worked out from a presentation of the fundamental group of the knot complement or using SnapPy.) Notice that adding the relations $a=1$ and $b^3=1$ to the presentation gives a presentation for the dihedral group $D_3$ and so the group is not $\Z_2$. Hence by the same argument as in the proof of Lemma~\ref{lem2}, $(\Sigma'', \phi'')$ is not a stabilization.
\end{proof}

%
%

\section{Destabilizable open book decompositions of $S^3$}

This section is devoted to showing that all planar open books with 4 or fewer boundary components that support the standard tight contact structure on $S^3$ destabilize. 

\begin{proof}[Proof of Theorem~\ref{thm2}]
We begin by making some general observations about planar open book decompositions for the tight contact structure on $S^3$.
Suppose $(\Sigma, \phi)$ is such a planar open book decomposition. Let $\partial \Sigma=c_0\cup c_1\cup\ldots\cup c_n$, $n\geq0$. We will think of $\Sigma$ as a disk with $n$ disjoint sub-disks removed and think of $c_0$ as the boundary of the original disk. 

Since $(S^3, \xi_{std})$ is Stein fillable we can use Corollary~2 in \cite{Wendl10} to see that the monodromy $\phi$ has a positive factorization $\tau_{\gamma_{1}}\circ\tau_{\gamma_{2}}\circ\ldots\circ\tau_{\gamma_{m}}$, where $\gamma_{i}$, $i=1,2, \ldots, m$, is an essential simple closed curve in $\Sigma$. From this positive factorization, we can construct a Lefschetz fibration of a Stein filling of $(S^3, \xi_{std})$ whose Euler characteristic is $2-(n+1)+m$, see \cite{Giroux02}.  According to \cite{Eliashberg90b}, $D^{4}$ is the unique Stein filling of $(S^3, \xi_{std})$. Since $\chi(D^4)=1$ we have $2-(n+1)+m=1$, and hence, $m=n$.

If $\gamma_{i}$, $1\leq i\leq n$, encloses the boundary components $c_{i_1}, \ldots, c_{i_k}$, where $i_1, \ldots, i_k\in \{1, 2, \ldots, n\}$, then denoting the homology class of a curve $c$ by $[c]$ we see that 
\[
[\gamma_i]=[c_{i_1}]+ \ldots +[c_{i_k}]
\]
in $H_1(\Sigma)$. Let $M$ be the $n\times n$ matrix whose $i,j$ entry is 1 if $[c_j]$ is part of the homology expansion of $\gamma_i$ and 0 otherwise.  Thus we see that $[\gamma_i]=\sum_j m_{i,j}[c_j]$. One may easily check that a presentation for the first homology of the 4--manifold described as a Lefschetz fibration by $(\Sigma,\phi)$, and the factorization of $\phi$, is 
\[
\langle
[c_1],\ldots, c_n]| [\gamma_1],\ldots, [\gamma_n]
\rangle.
\]
Thus the homology will be trivial if and only if $\det M=\pm1$.  

We now return to the proof of the theorem. If $\Sigma$ has 2 or 3 boundary components then it is clear that it must destabilize. So we are left to consider the case when $\Sigma$ has 4 boundary components, that is when $n=3$. By the above argument, there are $3$ essential simple closed curves $\gamma_{1}, \gamma_{2}, \gamma_{3}$ such that $\phi=\tau_{\gamma_{1}}\circ\tau_{\gamma_{2}}\circ\tau_{\gamma_{3}}$. Moreover there is a $3\times 3$ matrix $M$ with entries 0's and 1's representing the homological relations between the $\gamma_i$ and $c_j$. In addition $\det M=\pm1$. 

Notice that since $\Sigma$ has only 4 boundary components, each $\gamma_i$ is either parallel to one of the boundary components or bounds a region containing exactly two boundary components. 

\smallskip
\noindent
\textbf{Case 1}: each $\gamma_i$ is boundary parallel.  We get a destabilizing arc by taking an arc from the boundary compoent not parallel to a $\gamma_i$ to one of the other boundary components.

\smallskip
\noindent
\textbf{Case 2}: there are exactly two $\gamma_i$ which are boundary parallel. With out loss of generality we suppose $\gamma_{1}$ and $\gamma_{2}$ are parallel to $c_1$ and $c_2$, respectively. We notice that $\gamma_{3}$ cannot enclose $c_1$ and $c_2$, since if it does $\det M\neq\pm1$. So $\gamma_{3}$ enclose $c_3$ and one of $c_1$ and $c_2$, say $c_1$.  We can now find a stabilizing arc connecting $c_1$ to $c_3$.

\smallskip
\noindent
\textbf{Case 3}: each $\gamma_i$ is not boundary parallel. Then each $\gamma_i$ encloses $2$ boundary components. Either two $\gamma_i$ enclose the same two boundary components, or $\gamma_{1}, \gamma_{2}, \gamma_{3}$ enclose $3$ different pairs of boundary components. In either cases, $\det M\neq\pm1$.

\smallskip
\noindent
\textbf{Case 4}: there is exactly one $\gamma_i$ which is boundary parallel. Without loss of generality we suppose that $\gamma_1$ is parallel to the boundary component $c_1$. In order for $\det M$ to equal $\pm1$ we cannot have $\gamma_2$ and $\gamma_3$ enclosing the same boundary components. Up to diffeomorphism (and relabeling the boundary components), we can assume that $\gamma_2$ is a convex curve enclosing $c_1$ and $c_2$, see Figure~\ref{fig:ob4}.
\begin{figure}[htb]
\begin{overpic}
{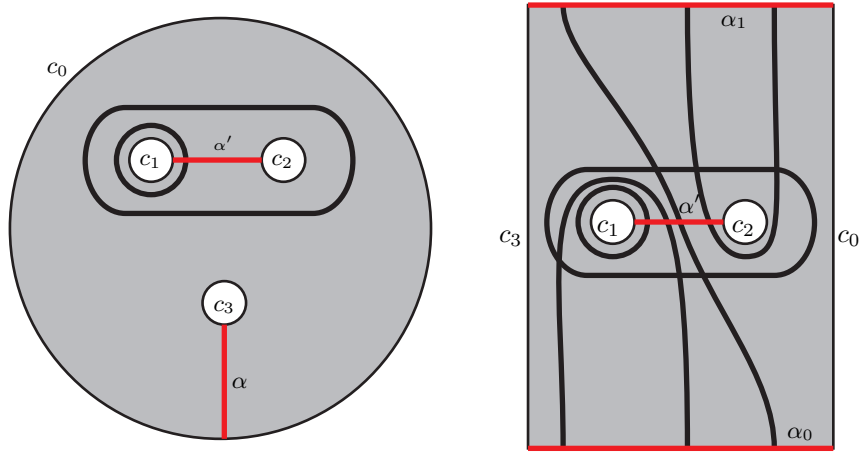}
{\small \put(15, 144){$c_0$}
\put(50,109){$c_1$}
\put(100,109){$c_2$}
\put(78, 54){$c_3$}
{\tiny \put(78, 115){$\alpha'$}}
\put(85,25){$\alpha$}
\put(295, 6){$\alpha_0$}
\put(270, 162){$\alpha_1$}
\put(254,91){$\alpha'$}}
\put(223,84){$c_1$}
\put(274, 84){$c_2$}
\put(314, 80){$c_0$}
\put(186, 80){$c_3$}
\end{overpic}
\caption{The surface $\Sigma$ with boundary components labeled and the curves $\alpha, \alpha', \gamma_1$ and $\gamma_2$ on the lefthand side. On the righthand side is the surface cut open with a potential $\gamma_3$ curve.}
\label{fig:ob4}
\end{figure}
We will assume that $\gamma_3$ encloses $c_2$ and $c_3$ (the only other possibility is that it encloses $c_1$ and $c_3$, but in this case either $\gamma_3$ intersects $\alpha$ one time and we can destabilize the open book or it does not and an argument exactly like the one below will show that the resulting manifold is not $S^3$).

Let $\alpha$ be an arc in $\Sigma$ which connects $c_0$ and $c_3$ and is disjoint from $\gamma_2$, see Figure~\ref{fig:ob4}.  If $\gamma_3$ intersects $\alpha$ once then it is a stabilizing arc. We are left to consider the case when $\gamma_3$ intersects $\alpha$ more than once (we note that since $\gamma_3$ is homologous to $[c_1]+[c_3]$ the algebraic intersection with $\alpha$ is 1). We claim in this case that $(\Sigma,\phi)$ is not an open book for $S^3$. To see this notice that if $\phi'$ is $\phi$ without the Dehn twist about $\gamma_3$ then $(\Sigma,\phi')$ is an open book decomposition for $S^2\times S^1$ and $\gamma_3$, sitting on one of the pages of the open book, is a knot in $S^2\times S^1$, see Figure~\ref{fig:s2s1}. Moreover, $(-1)$--surgery on $\gamma_3$ (here we mean that we are doing surgery with framing 1 less than the page framing) gives the manifold described by the open book $(\Sigma, \phi)$. Recall that Gabai \cite{Gabai87} proved that the Property R conjecture is true, that is the only knot in $S^2\times S^1$ that can be surgered to yield $S^3$ is the ``trivial knot" $\{p\}\times S^1$. In particular if $K$ is a knot in $S^2\times S^1$ which can be surgered to give $S^3$ then its complement has fundamental group $\Z$. We will complete the proof of our claim by showing that the complement of $\gamma_3$ thought of as a knot in $S^2\times S^1$ as above does not have fundamental group $\Z$.  

Consider the arc $\alpha'$ that connects $c_1$ to $c_2$. Isotope $\gamma_3$ on $\Sigma$ to intersect $\alpha$ and $\alpha'$ minimally. To understand how $\gamma_3$ sits in $S^2\times S^1$ we proceed as follows. Let $\Sigma'$ be $\Sigma$ cut open along $\alpha$. In the boundary of $\Sigma'$ there are two copies of $\alpha$ which we call $\alpha_0$ and $\alpha_1$. Notice that $\gamma_3$ will be cut into a collection of arcs in $\Sigma'$. Each such arc is of one of three kinds. Its endpoints are either both on $\alpha_0$, both on $\alpha_1$ or and one end point on each. We call these Type~I, II, and~III arcs, respectively. Notice there must be an odd number of Type~III arcs due to $\gamma_3$'s homological intersection with $\alpha$. Also notice that each Type~I and~II arc must intersect $\alpha'$ exactly once. (This is clear since if one were disjoint then that arc would ``shield" $c_1$ and $c_2$ from all arcs of a different type and that will contradict the minimality of the intersection of $\gamma_3$ with $\alpha$. Moreover if an arc intersected more than once then it will contradict the minimality of the intersection of $\gamma_3$ with $\alpha'$.) We also notice that a Type~III arc must also intersect $\alpha'$ exactly once since if it were disjoint then some of the Type~I or~II arcs would not intersect $\alpha'$ contradicting the above observation and if it intersected more than once we would again violate the minimality of the intersection of $\gamma_3$ and $\alpha'$. 

Let $S_\alpha$ be the non-separating 2-sphere in $S^2\times S^1$ defined by $\alpha$ (see the proof of Lemma~\ref{lem:de0} for the construction of $S_\alpha$). Let $X$ denote the complement of $\gamma_3$ in $S^2\times S^1$ and $F_\alpha=S_\alpha\cap X$. We claim that $F_\alpha$ is an incompressible surface in $X$. Since $F_\alpha$ is planar and has at least 3 boundary components the fundamental group of $F_\alpha$ is a free group on at least 2 generators. The incompressibility implies that $\pi_1(X)$ contains this free group and hence is not $\Z$. Thus our claim is complete once we see that $F_\alpha$ is incompressible. 

Cutting $S^2\times S^1$ open along $S_\alpha$ yields $S^2\times [0,1]$ and is depicted on the righthand side of Figure~\ref{fig:s2s1} along with some representative Type~I, II and~III arcs. While the arcs do not have to look exactly as in the figure the salient feature is that, due to the discussion above, they will all go through the leftmost box in the figure indicating the $(-1)$--twisting. 
\begin{figure}[htb]
\begin{overpic}
{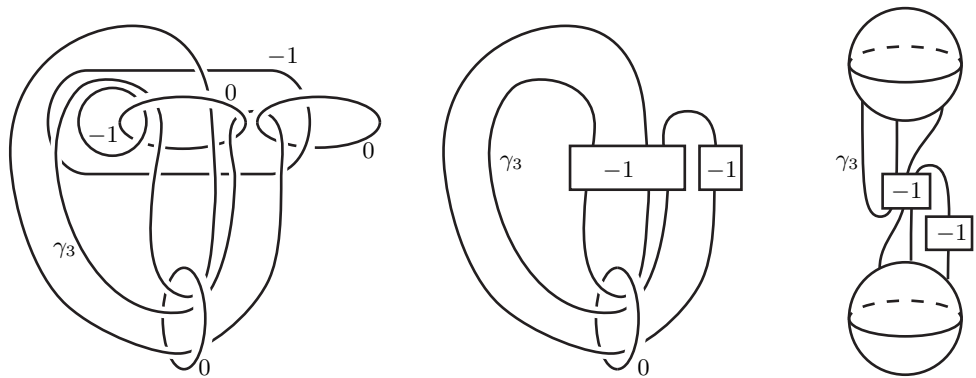}
\small{
\put(353, 52){$-1$}
\put(336, 68){$-1$}
\put(315, 80){$\gamma_3$}
\put(188, 80){$\gamma_3$}
\put(227, 76){$-1$}
\put(266, 76){$-1$}
\put(240, 1){$0$}
\put(74,1){$0$}
\put(136, 83){$0$}
\put(84,105){$0$}
\put(100,119){$-1$}
\put(32,89){$-1$}
\put(19,47){$\gamma_3$}
}\end{overpic}
\caption{The 3--manifold $S^2\times S^1$ described by the open book $(\Sigma,\phi')$ along with a sample $\gamma_3$ on the lefthand side. The same manifold in the center and on the right hand side this manifold is cut open along $S_\alpha$. Each box indicates that a complete lefthanded twist has been added to the strands going through the box.}
\label{fig:s2s1}
\end{figure}

Now let $D$ be a compressing disk for $F_\alpha$ in $X$. We can cut $X$ along $F_\alpha$ to get a manifold $Y\subset S^2\times [0,1]$ with two copies of $F_\alpha$ in its boundary, we denote them by $F_0$ and $F_1$. Notice that $D$ will be a disk in $Y$ with boundary on either $F_0$ or $F_1$, with out loss of generality assume it is on $F_0$. There is a disk $D'$ in $S_\alpha$  such that $D\cup D'$ bounds a ball in $S^2\times S^1$ (and in $S^2\times S^1$ cut open along $S_\alpha$). Notice that if a Type~I arc has one endpoint in $D'$, respectively $S_\alpha\setminus D'$, then the other endpoint has to be in $D'$, respectively $S_\alpha\setminus D'$, too since the arc must be disjoint form $D$. Similarly all the endpoints of Type~III arcs must be in $S_\alpha\setminus D'$ since they too must be disjoint from $D$. Finally we notice that some of the endpoints of Type I arcs must be in $D'$ or else $D$ is not a compressing disk for $F_\alpha$. But this is a contradiction. Indeed glue two 3--balls to $S^2\times S^1\setminus S_\alpha$ to get $S^3$. We can close each Type~I and~II arc in these added 3--balls to get unknots and connect the endpoints of each Type~III arc by a curve disjoint from all the arcs in the righthand diagram of Figure~\ref{fig:s2s1}. We can also push $D'$ slightly into one of the 3--balls so that it is disjoint form the added arcs. Now $D\cup D'$ is an embedded sphere in $S^3$ that bounds a ball containing some of the unknots associated to the Type~I arcs and does not contain any of the unknots associated to the Type~III arcs. Thus these respective unknots must be unlinked, but it is clear form the picture that they have linking $\pm 1$. 
\end{proof}


\def\cprime{$'$} \def\cprime{$'$}
\providecommand{\bysame}{\leavevmode\hbox to3em{\hrulefill}\thinspace}
\providecommand{\MR}{\relax\ifhmode\unskip\space\fi MR }
\providecommand{\MRhref}[2]{%
  \href{http://www.ams.org/mathscinet-getitem?mr=#1}{#2}
}
\providecommand{\href}[2]{#2}

\end{document}